\documentclass[a4paper, 11pt]{article}
\usepackage{amsmath}
\usepackage{amsfonts}
\usepackage{amssymb}
\usepackage[english]{babel}
\usepackage{graphicx}
\usepackage{amsthm,cases}

\usepackage[title]{appendix}

\usepackage{longtable}
\usepackage{tabularx}
\usepackage{diagbox,ulem}
\allowdisplaybreaks

\newtheorem{theorem}{Theorem}[section]
\newtheorem{definition}[theorem]{Definition}
\newtheorem{lemma}[theorem]{Lemma}
\newtheorem{corollary}[theorem]{Corollary}
\newtheorem{example}[theorem]{Example}
\newtheorem{construction}[theorem]{Construction}
\newtheorem{remark}[theorem]{Remark}

\def\whitebox{{\hbox{\hskip 1pt
 \vrule height 6pt depth 1.5pt
 \lower 1.5pt\vbox to 7.5pt{\hrule width
    3.2pt\vfill\hrule width 3.2pt}%
 \vrule height 6pt depth 1.5pt
 \hskip 1pt } }}
\def\qed{\ifhmode\allowbreak\else\nobreak\fi\hfill\quad\nobreak
     \whitebox\medbreak}

\newcommand{\ignore}[1]{}
\allowdisplaybreaks

\begin{document}

\title{Some results on generalized strong external difference
families}

\author{\small  X. Lu, X. Niu and H. Cao \thanks{Research
supported by the National Natural Science Foundation of China
under Grant 11571179, and the Priority Academic
Program Development of Jiangsu Higher Education Institutions.
E-mail: {\sf caohaitao@njnu.edu.cn}} \\
\small Institute of Mathematics, \\ \small  Nanjing Normal
University, Nanjing 210023, China}

\date{}
\maketitle
\begin{abstract}

A generalized strong external difference family (briefly $(v, m; k_1,\dots,k_m; \lambda_1,\dots,\lambda_m)$-GSEDF) was introduced by Paterson and Stinson in 2016. In this paper, we construct some new GSEDFs for $m=2$ and use them to obtain some results on graph decomposition. We also give some nonexistence results for GSEDFs. Especially, we prove that a $(v, 3;k_1,k_2,k_3; \lambda_1,\lambda_2,\lambda_3)$-GSEDF does not exist when $k_1+k_2+k_3< v$.

\bigskip


\end{abstract}

\section{Introduction}
 Let $G$ be an abelian group  of order $v$. For any two disjoint sets $D_1$, $D_2\subseteq G$, define $\Delta(D_1, D_2)=\{x-y: x\in D_1, y\in D_2\}$. Let $D$ be a $k$-subset of $G$, and let $\lambda$ and $\mu$ be  positive integers.
Then $D$ is called a $(v,k,\lambda)$-{\it difference set} (briefly $(v,k,\lambda)$-DS) in $G$ if $\Delta(D,D)=k\{0\}+\lambda(G\setminus\{0\})$,
and $D$ is called a $(v,k,\lambda,\mu)$-{\it partial difference set} (briefly $(v,k,\lambda,\mu)$-PDS) in $G$ if
$\Delta(D,D)=k\{0\}+\lambda (D\setminus\{0\})+\mu(G\setminus (D\cup\{0\}))$.
\begin{definition}
  Let $G$ be an abelian group of order $v$. Let $\lambda_1,\lambda_2,\dots,\lambda_m$ be positive integers and let $D_1,D_2,\dots,D_m$ be mutually disjoint subsets of $G$ such that $|D_i|=k_i$, $1\leq i\leq m$.  Then $\{D_1, D_2,\dots,D_m\}$ is called a  $(v,m;k_1,\dots,k_m;\lambda_1,\dots,\lambda_m)$-generalized strong external difference family (briefly $(v,m;k_1,\dots,k_m;\lambda_1,\dots,\lambda_m)$-GSEDF) in $G$ if the following multiset equation holds for each $i$, $1\leq i\leq m$,
\begin{equation*}
  \bigcup\limits_{1\leq j\leq m\atop j\neq i}\Delta(D_i,D_j)=\lambda_i(G\setminus \{0\}).
\end{equation*}
\end{definition}

It is easy to see that the parameters of a $(v,m;k_1,\dots,k_m;\lambda_1,\dots,\lambda_m)$-GSEDF satisfy the following counting relation for each $i$, $1\leq i\leq m$,
\begin{equation}\label{def}
\sum_{1\leq j\leq m\atop j\neq i}k_ik_j=\lambda_i (v-1).
\end{equation} A $(v,m;k_1,\dots,k_m;\lambda_1,\dots,\lambda_m)$-GSEDF is said to be a $(v,m,k,\lambda)$-SEDF when $k_1=\dots=k_m=k$ and $\lambda_1=\dots=\lambda_m=\lambda$.

Algebraic manipulation detection codes (briefly AMD codes) have many applications \cite{CDFPW, CFP, CPX} and GSEDFs can be used to obtain $R$-optimal strong AMD codes \cite{PS}. Therefore, it is important and interesting to determine whether or not there exist GSEDFs. Paterson and Stinson \cite{PS} prove that there  exists a $(v, m, k, 1)$-SEDF if and only if $m=2$ and $v=k^2+1$, or $m=v$ and $k=1$. Huczynska and Paterson \cite{HP} show that for a $(v,m, k, \lambda)$-SEDF, either  $k=1$ and $\lambda=1$, or
 $k>1$ and $\lambda<k$. They also show that a $(v,m, k, 2)$-SEDF can exist only  $m=2$, and
a $(v,2, p, \lambda)$-SEDF, where $p$ is prime, can exist only  $\lambda=1$. Wen, Yang, Fu and Feng \cite{WYFF} present some general constructions of GSEDF by using difference sets and partial difference sets.  There are some $(v,2,k,\lambda)$-SEDFs obtained from cyclotomic constructions, see \cite{BJWZ,HP,PS}.  Wen, Yang and Feng \cite{WYF}, and Jedwab and Li \cite{JL} respectively give an example of $(243,11,22,20)$-SEDF in two different ways which is the first nontrivial example for $m\geq5$.

Martin and Stinson \cite{MS}, and Jedwab and Li \cite{JL} use different methods to prove that if $k>1$, then there do not exist $(v,3,k,\lambda)$-SEDFs and $(v,4,k,\lambda)$-SEDFs in any finite abelian group. Further, Jedwab and Li \cite{JL} gave some upper bounds for a $(v,m,k,\lambda)$-SEDF, and used them to get some nonexistence  results for the case $m=2$. For more nonexistence results on $(v,m,k,\lambda)$-SEDFs, see \cite{BJWZ,HP,JL,MS}.

In this paper, we shall focus on the constructions and nonexistence of GSEDFs. In the next section, we will give some necessary definitions and notations, and some properties of GSEDFs. In Section 3, we give some nonexistence results for GSEDFs. Especially, we prove that a $(v, m;k_1,k_2,k_3; \lambda_1,\lambda_2,\lambda_3)$-GSEDF does not exist when $k_1+k_2+k_3< v$. In Section 4, we will construct some new GSEDFs for $m=2$ and use them to obtain some results on graph decomposition.  Finally, Section 5 gives some remarks and concludes this paper.

\section{Preliminaries}

The following theorem gives the relationship between difference sets and GSEDFs with the property that $\{D_1, D_2, \dots, D_m\}$ is  a partition of  $G$.

\begin{theorem}{\rm (\cite{PS})}\label{DS}
Suppose $\{D_1, D_2, \dots, D_m\}$ is a partition of an abelian group $G$ of order $v$, where $|D_i|=k_i$ for $1\leq i\leq m$. Then $\{D_1, D_2, \dots, D_m\}$ is a $(v,m;k_1,\dots,k_m;\lambda_1,\dots,\lambda_m)$-GSEDF if and only if each $D_i$ is a $(v,k_i,k_i-\lambda_i)$-DS in $G$, $1\leq i\leq m$.
\end{theorem}

In this paper, we shall focus on  the GSEDFs with the property that $\{D_1, D_2, \dots, D_m\}$ is not a partition of  $G$. The following theorem gives another  necessary and sufficient conditions for  the existence of a  GSEDF with the property that $\{D_1, D_2, \dots, D_m\}$ is  a partition of $G\setminus\{0\}$.

\begin{theorem}{\rm (\cite{HP})}\label{PDS}
Let $G$ be an abelian group of order $v$. Suppose $\{D_1, D_2, \dots, D_m\}$ is a partition of  $G\setminus\{0\}$, where $|D_i|=k_i$ for $1\leq i\leq m$. Then $\{D_1, D_2, \dots, D_m\}$ is a $(v,m;k_1,\dots,k_m;\lambda_1,\dots,\lambda_m)$-GSEDF if and only if each $D_i$ is a $(v,k_i,k_i-\lambda_i-1,k_i-\lambda_i)$-PDS in $G$,  $1\leq i\leq m$.
\end{theorem}

In the next section, we will use algebra tools  to prove the nonexistence of some classes of GSEDFs. Now we introduce some definitions and notations in group theory.  Let $\mathbb{Z}[G]$ be a group ring where $G$ is an abelian group of order $v$ with identity 1. For a given subset $D$ of $G$, we denote the group ring element $\sum_{d\in D}d$ by $D$ (by a standard abuse of notation), and the group ring element $\sum_{d\in D}d^{-1}$ by $D^{-1}$. Below  is an equivalent definition of a GSEDF in the form of a group ring $\mathbb{Z}[G]$.

Let $m\geq2$ and $D_1,D_2,\dots,D_m$ be mutually disjoint subsets of an abelian group $G$ with identity $1$, where $|G|=v$ and $|D_i|=k_i$ for $1\leq i\leq m$. Then $\{D_1,D_2,\dots,D_m\}$ is
a $(v,m;k_1,\dots,k_m;\lambda_1,\dots,\lambda_m)$-GSEDF in $G$ if and only if
\begin{equation}\label{1.2}
\sum_{1\leq j\leq m\atop j\neq i}D_iD_j^{-1}=\lambda_i(G\setminus \{1\})
\end{equation}
in $\mathbb{Z}[G]$ for each $i$, $1 \leq i \leq m$.

For a finite abelian group $G$, there are exactly $|G|$ distinct homomorphisms from $G$ to the multiplicative group of complex numbers \cite{L}. We called them {\it characters} of $G$ and they form a group. The group of characters is isomorphic to $G$. So we can label the $|G|$ distinct characters $\{\chi_a: a\in G\}$. Let $\widehat{G}$ denote the character group of an abelian group $G$, and let $\chi_0 \in \widehat{G}$ be the principal character. Each character $\chi \in \widehat{G}$ is extended linearly to the group ring $\mathbb{Z}[G]$ by $$\chi(\sum\limits_{g\in G}a_gg)=\sum\limits_{g\in G}a_g\chi(g).$$

The method of the proof of the following theorem is similar to Lemma 2.2 in \cite{JL}. We omit the proof here for simplicity.

\begin{theorem}\label{G^{N}}
 Suppose $\{D_1, D_2,\dots, D_m\}$ is a $(v,m;k_1,\dots,k_m;\lambda_1,\dots,\lambda_m)$-GSEDF in a finite abelian group $G$ and let $D=\bigcup\limits_{i=1}^{m}D_i$. If $D\neq G$, then there exists at least one non-principal character $\chi$ of $G$ such that $\chi(D)\neq 0$.
\end{theorem}

A character $\chi$ of group $G$ may be nontrivial on $G$, but still annihilate a whole subgroup $D$ of $G$, in the sense that $\chi(d)=1$ for all $d\in D$. The set of all characters of $G$ annihilating a given subgroup $H$ is called the {\it annihilator} of $D$ in $\widehat{G}$.

\begin{theorem}{\rm (\cite{LH})}
Let $D$ be a subgroup of the finite abelian group of $G$. Then the annihilator of $D$ in $\widehat{G}$ is a subgroup of $\widehat{G}$ of order $|G|/|D|$.
\end{theorem}

From the above theorem, we can obtain the following corollary easily.

\begin{corollary}\label{coro01}
Let $D$ be the set of the nontrivial subgroup of the finite abelian group of $G$. Then there exists a nonprincipal character $\chi\in\widehat{G}$ such that $\chi(D)=|D|$.
\end{corollary}

For later use, we need the following two lemmas.

\begin{lemma}\label{GSEDFm}
Let $m\ge 3$ and $k=k_1+k_2+\dots+k_m$, where $k_1\leq k_2\leq \dots\leq k_m$. Suppose $\{D_1,D_2,\dots,D_m\}$ is a $(v,m;k_1,\dots,k_m;$ $\lambda_1,\dots,\lambda_m)$-GSEDF. Then we have\\ (1) $\lambda_1\leq\lambda_2\leq\dots\leq\lambda_m$;\\ (2) $\lambda_1+\lambda_2\dots+\lambda_{m-1}>\lambda_m$; and\\ (3) $\lambda_i\leq k_i$   if $k\leq v$, and
$\lambda_i< k_i$,  if $k< v$.
\end{lemma}

\begin{proof}
 By the equation (\ref{def}) we have  $\lambda_i(v-1)=k_i(k-k_i)$, $1\le i\le m$. Then we have
$\lambda_j(v-1)-\lambda_i(v-1)=k_j(k-k_j)-k_i(k-k_i)$, $1\le i< j\le m$. That is
$(\lambda_j-\lambda_i)(v-1)=(k_j-k_i)(k-k_j-k_i)$. Since $m\ge 3$ and $k_i\le k_j$, we know that $\lambda_j\ge \lambda_i$.
Thus $\lambda_1\leq\lambda_2\leq\dots\leq\lambda_m$.

 For the second case, we have $(\lambda_1+\lambda_2\dots+\lambda_{m-1}-\lambda_m)(v-1)=k_1(k-k_1)+k_2(k-k_2)+\dots+k_{m-1}(k-k_{m-1})-k_m(k-k_m)
=k_1(k-k_1)+k_2(k-k_2)+\dots+k_{m-1}(k-k_{m-1})-k_m(k_1+k_2+\dots+k_{m-1})=k_1(k-k_1-k_m)+k_2(k-k_2-k_m)+\dots+k_{m-1}(k-k_{m-1}-k_m)>0$. Then
 we get $\lambda_1+\lambda_2\dots+\lambda_{m-1}>\lambda_m$.

For the last case, from $\lambda_i(v-1)=k_i(k-k_i)$, $1\le i\le m$, we know that $\frac {\lambda_i} {k_i}=\frac {k-k_i} {v-1}$. Then the conclusion follows.
\end{proof}

\begin{lemma}\label{ineq}
Let $x_i>0$, $1\leq i\leq n+1$, $n\geq2$.\\
(1) If $x_1+x_2+\dots+x_n>x_{n+1}$, then $\sqrt{x_1}+\sqrt{x_2}+\dots+\sqrt{x_n}>\sqrt{x_{n+1}}.$ \\
(2) $\sqrt{1+x_1}+ \sqrt{1+x_2}+\dots+\sqrt{1+x_n}>n-1+\sqrt{1+x_1+x_2+\dots+x_{n}}.$
\end{lemma}

\begin{proof} The case (1) is obvious. Now we prove the case (2) by induction on $n$. It is easy to see  that
 $ \sqrt{1+x_1}+ \sqrt{1+x_2}>1+\sqrt{1+x_1+x_2}$ since
 $ \sqrt{(1+x_1)(1+x_2)}>\sqrt{1+x_1+x_2}$. So the statement is right for $n=2$.
Suppose that the statement holds while $n\leq k$ for any $k\geq 2$.
Then we have
\begin{align*}
&\sqrt{1+x_1}+\dots+\sqrt{1+x_{k}}+\sqrt{1+x_{k+1}}\\ >&k-1+\sqrt{1+x_1+x_2+\dots+x_{k}}+\sqrt{1+x_{k+1}}\\
>&k-1+1+\sqrt{1+x_1+x_2+\dots+x_{k}+x_{k+1}}\\
=&k+\sqrt{1+x_1+x_2+\dots+x_{k}+x_{k+1}}
\end{align*}
So the statement is right for $n=k+1$. Thus the conclusion follows.
\end{proof}

\begin{theorem}
 Suppose  $m\geq2$ and $\{D_1, D_2,\dots,D_m\}$ is a $(v,m;k_1,\dots,k_m;\lambda_1,\dots,\lambda_m)$-GSEDF in a finite abelian group $G$. Let $D=\bigcup\limits_{i=1}^{m}D_i$. If $D\neq G$,  then  $gD_i$ can not be a  subset of any proper subgroup of $G$ for each $g\in G$ and each $i\in \{1,2,\dots,m\}$.
\end{theorem}

\begin{proof}
It is easy to check that $\{gD_1,gD_2,\dots,gD_m\}$ is a $(v,m;k_1,\dots,k_m;\lambda_1,\dots,\lambda_m)$-GSEDF since $\{D_1,D_2,\dots,D_m\}$ is a $(v,m;k_1,\dots,k_m;\lambda_1,\dots,\lambda_m)$-GSEDF. So we only need to prove the conclusion is right when $g$ is the unity of $G$.
 If  $D_i$ is a subset of some proper subgroup of $G$, then by Corollary \ref{coro01} there exists a nonprincipal character $\chi$ such that $\chi(D_i)=k_i$. Then we have $k_i(\overline{\chi(D)}-k_i)=-\lambda_i$. Thus $\chi(D)=k_i-\frac {\lambda_i} {k_i}$ is a rational number. Since $\chi(D)$ is also an algebraic integer, we know that $\frac {\lambda_i} {k_i}$ is an integer.  So we have $\lambda_i\geq k_i$. This contradicts Lemma \ref{GSEDFm} (3) since $D\neq G$.
\end{proof}

\section{Nonexistence of GSEDFs}

In this section, we shall give some nonexistence results of GSEDFs. We start with two conclusions from number theory.

\begin{theorem}\label{prime}
 Let  $p$ be a prime and $k_1+\dots+k_m\le p$, $m\geq 2$. Then there does not exist a $(p+1,m;k_1,\dots,k_m;\lambda_1,\dots,\lambda_m)$-GSEDF.
\end{theorem}
\begin{proof}
 Let $k=k_1+\dots+k_m$. If there  exists a $(p+1,m;k_1,\dots,k_m;\lambda_1,\dots,\lambda_m)$-GSEDF, then  we  have  $\lambda_ip=k_i(k-k_i)$, $1\le i\le m$. So $p\mid k_i$ or $p\mid k-k_i$ since $p$ is a prime. It is a contradiction to $p>k_i$ or $p>k-k_i$ which can be obtained from $k\le p$.
\end{proof}

\begin{theorem}
 Let $p_1$ and $p_2$ be different primes, and $k_1+\dots+k_m\le p_1p_2$, $m\geq 3$. Then there does not exist a $(p_1p_2+1,m;k_1,\dots,k_m;\lambda_1,\dots,\lambda_m)$-GSEDF.
\end{theorem}

\begin{proof}
Let $k=k_1+\dots+k_m$. If there exists a $(p_1p_2+1,m;k_1,\dots,k_m;\lambda_1,\dots,\lambda_m)$-GSEDF, then $\lambda_ip_1p_2=k_i(k-k_i)$, $1\le i\le m$. For each $i$, if $p_1\nmid k_i$ and $p_2\nmid k_i$, then $p_1|(k-k_i)$ and $p_2|(k-k_i)$, so $p_1p_2|(k-k_i)$ since $p_1$ and $p_2$ are different primes. It is  a contradiction to $k-k_i<p_1p_2$ since $m\geq 3$  and $k_1+\dots+k_m\le p_1p_2$. Thus we have proved that either $p_1|k_i$ or $p_2|k_i$, $1\le i\le m$.

On the other hand, we can prove that if $p_1|k_i$, then  $p_2\nmid k_i$. Otherwise, $p_1p_2|k_i$ which is a contradiction to $k_i<p_1p_2$ since $m\geq 3$ and $k_1+\dots+k_m\le p_1p_2$. Similarly, we can prove that if  $p_2|k_i$, then  $p_1\nmid k_i$. So we have either $p_1|k_i$, $p_2\nmid k_i$ and $p_2|(k-k_i)$  or $p_2|k_i$, $p_1\nmid k_i$ and $p_1|(k-k_i)$.

Further, if $p_1|k_i$ holds for all $1\le i\le m$, then $p_2|(k-k_i)$ holds for all $1\le i\le m$, and $p_1|k$ holds. So $p_1|(k-k_i)$ holds for all $1\le i\le m$.  Thus, $p_1p_2|(k-k_i)$ holds for all $1\le i\le m$.  It is  a contradiction to $k-k_i<p_1p_2$. Thus, without lose of generality, we suppose that $p_1|k_i$ and $p_2|(k-k_i)$ for $1\le i\le l$, and $p_2|k_j$ and $p_1|(k-k_j)$ for $l+1\le j\le m$.

If $l\ge 2$, then we have $p_2|(k-k_1-\sum\limits_{j=l+1}^mk_j)$. That is $p_2|(\sum\limits_{i=1}^lk_i-k_1)$. Note that $p_1|(\sum\limits_{i=1}^lk_i-k_1)$. So we have $p_1p_2|(\sum\limits_{i=1}^lk_i-k_1)$. It is a contradiction. Similarly, if $l=1$, then we have $p_1|(k-k_{l+1}-\sum\limits_{i=1}^lk_i)$. That is $p_1|(\sum\limits_{j=l+1}^mk_j-k_{l+1})$. Note that $p_2|(\sum\limits_{j=l+1}^mk_j-k_{l+1})$. So we have $p_1p_2|(\sum\limits_{j=l+1}^mk_j-k_{l+1})$. It is a contradiction.
\end{proof}

For the next theorem, we need the following notations. Suppose $\{D_1, D_2,\dots, D_m\}$ is a $(v,m;k_1,\dots,k_m;\lambda_1,\dots,\lambda_m)$-GSEDF in a finite abelian group $G$ and let $D=\bigcup\limits_{i=1}^{m}D_i$. Then the equation (\ref{1.2}) is equivalent to
\begin{equation}\label{j}
D_jD^{-1}-D_jD_j^{-1}=\lambda_j(G\setminus\{1\})
\end{equation}
 in $\mathbb{Z}(G)$ for $1\leq j\leq m$.

For GSEDFs with $D\not=G$,  by Theorem \ref{G^{N}} there exists a nonprincipal character $\chi\in \widehat{G}$ such that $\chi(D)\neq 0$. So we apply this character on the equation (\ref{j}) to obtain \begin{equation}\label{D}
\chi(D_j)\overline{\chi(D)}-\chi(D_j)\overline{\chi(D_j)}=-\lambda_j.
\end{equation}

And from the equation (\ref{D}), we can get that $$\chi(D_j)=\frac {|\chi(D_j)|^{2}-\lambda_j} {|\chi(D)|^{2}}\chi(D).$$

Let $\frac {|\chi(D_j)|^{2}-\lambda_j} {|\chi(D)|^{2}}=\alpha_j$, then $\chi(D_j)=\alpha_j\chi(D)$. We now conjugate  the equation (\ref{D}) and obtain that $$\alpha_j^{2}-\alpha_j-\frac {\lambda_j} {|\chi(D)|^{2}}=0.$$

Let $c=\frac {4} {|\chi(D)|^{2}}$. Then the solutions of this equation are
\begin{equation}
  \alpha_j^{+}=\frac 1 2\left(1+\sqrt{1+c\lambda_j} \right),\ \ \  \alpha_j^{-}=\frac 1 2\left(1-\sqrt{1+c\lambda_j}\right).
\end{equation}

It is obvious that $\alpha_j^{+}>1$ and $\alpha_j^{-}<0$. Since $\chi(D)\neq 0$ and $\chi(D)=\sum\limits_{j=1}^m\chi(D_j)=\chi(D)\sum\limits_{j=1}^m\alpha_j,$
we have $\sum\limits_{j=1}^m\alpha_j=1$, where $\alpha_j\in\{\alpha_j^{+},\alpha_j^{-}\}$. For later use, we state them in the following lemma.

\begin{lemma}\label{equ}
  If there exists a $(v,m;k_1,\dots,k_m;\lambda_1,\dots,\lambda_m)$-GSEDF in a finite abelian group $G$ of order $v$ such that $\sum\limits_{i=1}^{m}k_i< v$, then $\sum\limits_{j=1}^m\alpha_j=1$ holds for some $\alpha_j$, where $\alpha_j\in\{\alpha_j^{+},\alpha_j^{-}\}$.
\end{lemma}

\begin{theorem}\label{GSEDF3}
  There does not exist a $(v,3;k_1,k_2,k_3;\lambda_1,\lambda_2,\lambda_3)$-GSEDF if $\sum\limits_{i=1}^{3}k_i< v$.
\end{theorem}

\begin{proof}
Without loss of generality, we may suppose that $k_1\leq k_2\leq k_3$. By Lemma \ref{GSEDFm} we obtain that $\lambda_1\leq\lambda_2\leq\lambda_3$ and $\lambda_i+\lambda_j>\lambda_k$, where $\{i, j, k\}=\{1,2,3\}$. By Lemma \ref{equ} we have $\alpha_1+\alpha_2+\alpha_3=1$, where $\alpha_j\in\{\alpha_j^{+},\alpha_j^{-}\}$, $j=1,2,3$. We distinguish the following four cases.
\begin{description}
  \item[Case 1:] $\alpha_i^{+}+\alpha_j^{+}+\alpha_k^{+}=1$.
  It is  impossible since $\alpha_i^{+}>1$, $1\le i\le 3$.

  \item[Case 2:] $\alpha_i^{-}+\alpha_j^{-}+\alpha_k^{-}=1$.
It is  impossible since $\alpha_i^{-}<0$, $1\le i\le 3$.

  \item[Case 3:] $\alpha_i^{+}+\alpha_j^{+}+\alpha_k^{-}=1$.
  Then we have
 \begin{equation}\label{4}
1+\sqrt{1+c\lambda_i}+\sqrt{1+c\lambda_j}-\sqrt{1+c\lambda_k}=0.
\end{equation}
By Lemma \ref{ineq} we have
$1+\sqrt{1+c\lambda_i}+\sqrt{1+c\lambda_j}>1+1+\sqrt{1+c\lambda_i+c\lambda_j}>2+\sqrt{1+c\lambda_k}$.
So  the left side of equation (\ref{4}) can not equal to 0.

  \item[Case 4:]  $\alpha_i^{+}+\alpha_j^{-}+\alpha_k^{-}=1$.
  Then we have
 \begin{equation}\label{6}
 1+\sqrt{1+c\lambda_i}-\sqrt{1+c\lambda_j}-\sqrt{1+c\lambda_k}=0.
\end{equation}
By Lemma  \ref{ineq} we know that
  $\sqrt{1+c\lambda_j}+\sqrt{1+c\lambda_k}>1+\sqrt{1+c(\lambda_j+\lambda_k)}>1+\sqrt{1+c\lambda_i}$.
  So  the left side of equation (\ref{6}) can not equal to 0.
\end{description}

Therefore, the conclusion follows as above.
\end{proof}

Now we  consider  the  existence of a $(v,3;k_1,k_2,k_3;\lambda_1,\lambda_2,\lambda_3)$-GSEDF with $\sum\limits_{i=1}^{3}k_i=v$.  It has been proved that there does not exist a  $(v,3,k,\lambda)$-SEDF. So we only need to consider the case $|\{k_1,k_2,k_3\}|\ge 2$.

 \begin{theorem}\label{k_2= k_3}
Suppose there exists a $(v,3;k_1,k_2,k_3;\lambda_1,\lambda_2,\lambda_3)$-GSEDF with $\sum\limits_{i=1}^{3}k_i=v$.  If $|\{k_1,k_2,k_3\}|=2$ or $1\in \{k_1,k_2,k_3\}$, then $v\equiv 3\pmod 4$, $(k_1,k_2,k_3)=(1,\frac {v-1} 2,\frac {v-1} 2)$ and $(\lambda_1,\lambda_2,\lambda_3)=(1, \frac {v+1} 4, \frac {v+1} 4)$.
\end{theorem}

\begin{proof}
If  $|\{k_1,k_2,k_3\}|=2$, without lose of generality, we may suppose $k_2=k_3$. So $k_1=v-k_2-k_3=v-2k_2$. Then from equation \eqref{def}, we have $(v-2k_2)(2k_2)=\lambda_1(v-1)$ and $\lambda_2(v-1)=k_2(v-k_2)$. Thus $v-1\mid (4k_2^{2}-2k_2)$ and $v-1\mid (k_2^{2}-k_2)$. So we obtain $v-1\mid 2k_2$. Therefore, $k_2=k_3=\frac {v-1} 2$ and $k_1=1$. So $(\lambda_1,\lambda_2,\lambda_3)=(1, \frac {v+1} 4, \frac {v+1} 4)$ and $v\equiv 3\pmod 4$.

If  $1\in \{k_1,k_2,k_3\}$,  without lose of generality, we may suppose $k_1=1$. Then we have $v-1=\lambda_1(v-1)$, $\lambda_2(v-1)=k_2(1+k_3)$ and $\lambda_3(v-1)=k_3(1+k_2)$. So we obtain $\lambda_1=1$ and $k_3-k_2 = (\lambda_3-\lambda_2)(v-1)$. Thus $v-1\mid (k_3-k_2)$ which leads to $k_2=k_3$. Then the conclusion follows from the above case.
\end{proof}

 \begin{theorem}\label{prime power}
  If  $v\equiv 3\pmod 4$ is a prime power, then there is a $(v,3;1,\frac {v-1} 2,\frac {v-1} 2;1, \frac {v+1} 4, \frac {v+1} 4)$-GSEDF.
 \end{theorem}
\begin{proof}
 Let $\mathbb{F}_v^{\ast}=<g>$, $D_1=\{0\}$, $D_2=\{g^{2i}: 0\leq i\leq \frac {v-1} 2\}$ and $D_3=\{g^{2i+1}: 0\leq i\leq \frac {v-1} 2\}$. Then $D_1$ and $D_2$ are respectively a $(v,1,0)$-DS and  a $(v, \frac {v-1} 2, \frac {v-3} 4)$-DS in $\mathbb{F}_v$ \cite{S}.  $D_3$ is also a $(v, \frac {v-1} 2, \frac {v-3} 4)$-DS in $\mathbb{F}_v$ since $D_3=gD_2$. It is easy to know $\{D_1, D_2, D_3\}$ is a partition of $\mathbb{F}_v$. Then by Theorem \ref{DS} $\{D_1, D_2, D_3\}$ is a $(v,3;1,\frac {v-1} 2,\frac {v-1} 2;1, \frac {v+1} 4, \frac {v+1} 4)$-GSEDF.
\end{proof}

By Theorem \ref{k_2= k_3} we know that a $(v,3;k_1,k_2,k_3;\lambda_1,\lambda_2,\lambda_3)$-GSEDF with $\sum\limits_{i=1}^{3}k_i=v$ can exist only when (1) $v\equiv 3\pmod 4$ and $(k_1,k_2,k_3,\lambda_1,\lambda_2,\lambda_3)=(1,\frac {v-1} 2,\frac {v-1} 2,1, \frac {v+1} 4, \frac {v+1} 4)$, or (2) $1<k_1<k_2<k_3$. Actually, the case $1<k_1<k_2<k_3$ can be enhanced to $\sqrt{v}<k_1<k_2<k_3$. Next we will discuss the existence of a $(v,3;k_1,k_2,k_3;\lambda_1,\lambda_2,\lambda_3)$-GSEDF for some small values of $v$.

For the first case $v\equiv 3\pmod 4$, it is easy to see $v=3$ is trivial.
We partition the first 24 possible values of $3<v<100$ into 3 sets $M_1, M_2, M_3$, where
$M_1= \{15, 35, 63, 99\}$,
$M_2= \{39, 51,55, 75, 87, 91, 95\}$
and $M_3= \{7, 11, 19, 23, 27, 31, 43, 47,  59, 67, 71, 79, 83\}$.
By Theorem \ref{prime power}, a $(v,3;1,\frac {v-1} 2,\frac {v-1} 2;1, \frac {v+1} 4, \frac {v+1} 4)$-GSEDF exists for each $v\in M_3$. By Theorem \ref{DS}, a $(v,3;1,\frac {v-1} 2,\frac {v-1} 2;1, \frac {v+1} 4, \frac {v+1} 4)$-GSEDF can not exist for any $v\in M_2$ since there is no $(v, \frac {v-1} 2, \frac {v-3} 4)$-DS \cite{B}. Thus, for this case, there are 4 possible values of $v\in M_1$ for which the existence of a $(v,3;1,\frac {v-1} 2,\frac {v-1} 2;1, \frac {v+1} 4, \frac {v+1} 4)$-GSEDF remains open.

For the second case $\sqrt{v}<k_1<k_2<k_3$, we list all the possible parameters $(v,k_1,k_2,k_3)$ of a $(v,3;k_1,k_2,k_3;\lambda_1,\lambda_2,\lambda_3)$-GSEDF with $v\leq 200$ and $\sum\limits_{i=1}^{3}k_i=v$ as follows:
\[ \begin{array}{rrrrrr}
(31,6,10,15), & (43,7,15,21),\ & (67,12,22,33),&  (71,15,21,35),&  (79,13,27,39),\\ (85,21,28,26), & (91,10,36,45),&  (103,18,34,51), &  (106,15,21,70),&  (111,11,45,55),\\  (115,19,39,57),& (127,28,36,63),&  (131,26,40,65),&  (133,12,33,88),&  (139,24,46,69),\\ (151,25,51,75),&  (155,22,56,77),&  (166,45,55,66),&  (171,35,51,85),&  (175,30,58,87),\\ (181,36,45,100),&  (183,14,78,91),&  (187,31,63,93),&  (191,20,76,95),&  (199,45,55,99).
\end{array}\]
By Theorem \ref{DS}, none of the above GSEDF exists in finite abelian group since at least one corresponding difference set of each GSEDF do not exist. So we only need to show the nonexistence of the corresponding difference sets. There are  no $(171,35,7)$-DS  \cite{La} and $(175,87,43)$-DS \cite{B, Kopilovich} in any abelian group, and all the other corresponding difference sets are ruled out in \cite{B}.

\section{New constructions for GSEDFs with $m=2$}

In this section, we will present some constructions for GSEDFs with $m=2$. We also establish a relationship between GSEDFs and graph decompositions.

\begin{lemma}\label{m1}
 There does not exist a $(v,m;k_1,\dots,k_m;1,\dots,1)$-GSEDF where $m\geq3$ and $k_i>1$ holds for some $i$, $1\leq  i\leq m$.
\end{lemma}

\begin{proof}
Without lose of generality, we suppose $\{D_1, D_2,\dots, D_m\}$ is a $(v, m;k_1,\dots,k_m;1,\dots,1)$-GSEDF in $G$ with $m\geq 3$ and $k_1>1$.
Then we have
\begin{equation}\label{m-2}
  \bigcup\limits_{2\leq i,j\leq m\atop j\neq i}\Delta(D_i,D_j)=(m-2)(G\setminus \{0\}).
\end{equation}
Let $x$, $y\in D_1$ and $x\neq y$. Since $m\geq 3$ and from \eqref{m-2}, there exist $u\in D_i$ and $v\in D_j$, where $i$, $j>1$ and $i\neq j$ such that $u-v=x-y$. That is $u-x=v-y$, it is a contradiction to $\lambda_1=1$.
\end{proof}

\begin{theorem}\label{3.7}
 There exists a $(v,m;k_1,\dots,k_m;1,\dots,1)$-GSEDF if and only if $m=2$ and $v=k_1k_2+1$, or $k_i=1$ for $1\leq i\leq m$ and $v=m$.
\end{theorem}

\begin{proof}
By Lemma \ref{m1}, we only need to consider $m=2$ or $k_i=1$ for $1\leq i\leq m$. When $m=2$, from $1\times(v-1)=k_1k_2$, we have $v=k_1k_2+1$. When $k_i=1$ for $1\leq i\leq m$, from $\lambda_i(v-1)=(m-1)k_i^{2}$, we have $v=m$. So the conclusion follows.
\end{proof}

\begin{theorem}\label{equal}
 If there is a  $(v,2;k_1,k_2;\lambda_1,\lambda_2)$-GSEDF, then $\lambda_1=\lambda_2$. Further, if $k_1$ and $k_2$ are primes, then $\lambda_1=\lambda_2=1$.
\end{theorem}

\begin{proof}
 By the definition of GSEDF, we have
$\lambda_1(v-1)=k_1k_2=
\lambda_2(v-1)$.
 Then $\lambda_1=\lambda_2$ and $\lambda_1| k_1k_2$. And $\lambda_1<k_1$ and $\lambda_1<k_2$ since $k_1$ and $k_2$ are primes. Then  we obtain that $\lambda_1=\lambda_2=1$.
\end{proof}

\begin{theorem}\label{C1}
 There exists a $(ab+1,2;a,b;1,1)$-GSEDF.
 \end{theorem}

\begin{proof}
Let $G=\mathbb{Z}_{ab+1}$, $D_1 =\{0, 1,\ldots, a-1\}$, and  $D_2 =\{a, 2a, \ldots, ba\}$. It is easy  to check that
$\Delta(D_1,D_2)=\Delta(D_2,D_1)= \mathbb{Z}_{ab+1}\setminus \{0\}$. Then $\{D_1,D_2\}$ is a $(ab+1,2;a,b;1,1)$-GSEDF in $G$.
\end{proof}

By Theorems~\ref{3.7},~\ref{equal}, and \ref{C1} the existence of a $(v,2;k_1,k_2;1,1)$-GSEDF has been completely determined.  Now we continue to consider the existence of a $(v,2;k_1,k_2;\lambda,\lambda)$-GSEDF for $\lambda>1$. The following construction is the first recursive construction for GSEDFs.

\begin{construction}\label{construction1}
Let $v>1$, $t>1$ and $v\equiv t\equiv 1\pmod 2$.   If there is a $(v,2;2\lambda,\frac {v-1} 2;\lambda,\lambda)$-GSEDF, then there is a  $(vt,2;4\lambda,\frac {vt-1} 2;2\lambda,2\lambda)$-GSEDF.
\end{construction}

\begin{proof} Suppose $\{D_1,D_2\}$ is a $(v,2;2\lambda,\frac {v-1} 2;\lambda,\lambda)$-GSEDF in the group $G$ with identity $0$, where $|D_1|=2\lambda$, $|D_2|=\frac{v-1}{2}$, and  $\Delta(D_1,D_2)=\Delta(D_2,D_1)=\lambda(G\setminus \{0\})$. Let $$D_1'=\{(x,j):x\in D_1, j=0,1\},$$ $$D_2'=\{(x,0), (x,2i-1): x\in D_2, 1\leq i\leq \frac {t-1} 2\}\cup \{(y, 2i): y\in G\setminus D_2, 1\leq i\leq \frac {t-1} 2\}.$$  We will show that $\{D_1',D_2'\}$ is a  $(vt,2;4\lambda,\frac {vt-1} 2;2\lambda,2\lambda)$-GSEDF
in $G\times \mathbb{Z}_{t}$.

It is easy to check that $D_1'\subset G\times \mathbb{Z}_{t}$, $D_2'\subset G\times \mathbb{Z}_{t}$, $D_1'\cap D_2'=\emptyset$, $|D_1'|=2|D_1|=4\lambda$ and  $|D_2'|=(\frac {t-1} 2+1)|D_2|+(v-|D_2|)\frac {t-1} 2=\frac {vt-1} 2$. Also we have $\Delta(D_1',D_2')=\Delta_1\cup \Delta_2\cup \Delta_3$, where $$\Delta_1=\{(x,0),(x,-1):x\in \Delta(D_2,D_1)\},$$  $$\Delta_2=\{(x,2i-1),(x,2i-2):x\in \Delta(D_2,D_1), 1\leq i\leq \frac {t-1} 2 \},$$  $$\Delta_3=\{(y,2i),(y,2i-1):y\in \Delta(G\setminus D_2,D_1), 1\leq i\leq \frac {t-1} 2 \}.$$

Now we prove $\Delta(D_1',D_2')=2\lambda((G\times \mathbb{Z}_{t})\setminus \{(0,0)\})$.
In other words, we need to prove that every element $(a,b)\in (G\times \mathbb{Z}_{t})\setminus \{(0,0)\}$ appears exactly $2\lambda$ times in $\Delta(D_1',D_2')$. It is sufficient to prove that every element $(a,b)\in (G\times \mathbb{Z}_{t})\setminus \{(0,0)\}$ appears at least $2\lambda$ times since $|D_1'||D_2'|=2\lambda(vt-1)$.
We distinguish the following 3 cases.
\begin{description}
  \item[Case 1:] $b=0$. In this case, $a\not=0$ and $(a,0)$ appears $\lambda$ times in $\Delta_1$ and $\lambda$ times in $\Delta_2$ (let $i=1$), respectively, since $\Delta(D_2,D_1)=\lambda(G\setminus \{0\})$.
  \item[Case 2:] $b\equiv 1\pmod 2$. In this case, note that $\Delta(G\setminus D_2,D_1)\cup \Delta(D_2,D_1)=\Delta(G,D_1)=|D_1|G=2\lambda G$, then we know $(a,b)$ appears $\lambda$ times in $\Delta_2$ and $\lambda$ times in $\Delta_3$, respectively.
  \item[Case 3:] $b\equiv 0\pmod 2$ and $b\not=0$. When $2\le b\le t-3$, similar to Case 2, $(a,b)$ appears $\lambda$ times in $\Delta_2$ and $\lambda$ times in $\Delta_3$, respectively. When $b=t-1$, $(a,b)$ appears $\lambda$ times in $\Delta_1$ and $\lambda$ times in $\Delta_3$, respectively.
  \end{description}
So the conclusion follows.
\end{proof}

\begin{theorem}\label{power}
(1)(\cite{HP})  For any prime power  $q\equiv1\pmod 4$, there is a $(q, 2;\frac{q-1}{2},\frac{q-1}{2};\frac{q-1}{4},$ $\frac{q-1}{4})$-GSEDF.\\
(2)(\cite{WYFF}) For any prime power $q\equiv3\pmod4$, there is a  $(q,2;\frac{q-1}{2},\frac{q+1}{2};\frac{q+1}{4},\frac{q+1}{4})$-GSEDF.
\end{theorem}

\begin{theorem}\label{2n}
  Let $v=p_1p_2\dots p_n$, $p_i>1$, $1\le i\le n$, where $p_1, p_2,\dots,p_n$ are odd integers. Then there exists a $(v, 2;2^{n},\frac {v-1} 2;2^{n-1},2^{n-1})$-GSEDF.
\end{theorem}

\begin{proof}
 There is a $(p_1,2;2,\frac {p_1-1} 2;1,1)$-GSEDF by Theorem \ref{C1}. Applying Construction~\ref{construction1}, we obtain a $(v, 2;2^{n},\frac {v-1} 2;2^{n-1},2^{n-1})$-GSEDF by induction on $n$.
\end{proof}

\begin{theorem}\label{g16}
 There exists a $(16,2;5,9;3,3)$-GSEDF.
 \end{theorem}

\begin{proof}
Let $G=\mathbb{Z}_{2}\times\mathbb{Z}_{8}$, $D_1 =\{(0,0), (0,1), (0,3), (1,0), (1,4)\}$, and  $D_2 =\{(0,4), (0,5), (0,7), \\(1,1), (1,2), (1,3), (1,5), (1,6), (1,7)\}$. It is easy  to check that
$\Delta(D_1,D_2)=\Delta(D_2,D_1)= (\mathbb{Z}_{2}\times\mathbb{Z}_{8})\setminus \{(0,0)\}$. Then $\{D_1,D_2\}$ is a $(16,2;5,9;3,3)$-GSEDF in $G$.
\end{proof}

\begin{theorem}\label{q+2} Let $q$ and $q+2$ be odd prime powers and $v=q(q+2)$. Then there is a  $(v, 2;\frac{v-1}{2},\frac{v+1}{2};\frac{v+1}{4},\frac{v+1}{4})$-GSEDF.
\end{theorem}

\begin{proof}
Since $q$ and $q+2$ are odd prime powers, there is a $(v, \frac {v-1} 2, \frac {v-3} 4)$-DS, denoted by $D$,  in $G=\mathbb{F}_{q}\times\mathbb{F}_{q+2}$ \cite{SS}.  Then $G\setminus D$ is a $(v, \frac {v+1} 2, \frac {v+1} 4)$-DS in $G$. So by Theorem \ref{DS} $\{D, G\setminus D\}$ is a $(v, 2;\frac{v-1}{2},\frac{v+1}{2};\frac{v+1}{4},\frac{v+1}{4})$-GSEDF.
\end{proof}

\begin{theorem}\label{q4=1}Let $t=p_1p_2\dots p_n$, $p_i>1$, $1\le i\le n$, where $p_1, p_2,\dots,p_n$ are odd integers.
For any prime power $q$ with $q\equiv1\pmod 4$, there exists a $(qt, 2;(q-1)2^{n-1},\frac{qt-1}{2};(q-1)2^{n-2},(q-1)2^{n-2})$-GSEDF.
 \end{theorem}

\begin{proof}
 For any prime power  $q\equiv1\pmod 4$, there is a $(q, 2;\frac{q-1}{2},\frac{q-1}{2}; \frac{q-1}{4}, \frac{q-1}{4})$-GSEDF in $F_q$ by Lemma \ref{power} (1). Applying   Construction \ref{construction1} recursively for $n$ times, we can obtain a $(qt, 2;(q-1)2^{n-1},\frac{qt-1}{2};(q-1)2^{n-2},(q-1)2^{n-2})$-GSEDF.
\end{proof}

\begin{theorem}\label{4m-1} Let $t=1, n=0$; or $t=p_1p_2\dots p_n$, $p_i>1$, $1\le i\le n$, where $p_1, p_2,\dots,p_n$ are odd integers. If $4m-1$ is a prime power or $4m-1=q(q+2)$ where both $q$ and $q+2$ are prime powers, then
there is a $((4m-1)t,2;m\times 2^{n+1},\frac {4mt-t-1} 2;m\times 2^n,m\times 2^n)$-GSEDF.
\end{theorem}

\begin{proof}
When $t=1, n=0$, if $4m-1$ is a prime power or $4m-1=q(q+2)$ where both $q$ and $q+2$ are prime powers, then
there is a $(4m-1,2;2m,2m-1;m,m)$-GSEDF by Lemma \ref{power} (2) or Theorem \ref{q+2} respectively. When $t=p_1p_2\dots p_n$, $p_i>1$, $1\le i\le n$, where $p_1, p_2,\dots,p_n$ are odd integers, we can use  Construction \ref{construction1} to get a  $((4m-1)t,2;m\times 2^{n+1},\frac {4mt-t-1} 2;m\times 2^n,m\times 2^n)$-GSEDF.
\end{proof}

Now we show the relationship between a $(v,2;k_1,k_2;\lambda,\lambda)$-GSEDF and  the decomposition of complete multigraphs into complete bipartite graphs.
We briefly review some definitions about graphs. For more definitions of graph theory, see \cite{BM}.

Let $\Gamma$ be a graph and $\lambda \Gamma$ be the graph obtained by assigning each edge of $\Gamma$ a multiplicity $\lambda$. An $H$-decomposition  of a graph $\Gamma$
is a partition of the edge set of $\Gamma$ into $|E(\Gamma)|/|E(H)|$ subgraphs, each of which is isomorphic to $H$. An $H$-decomposition  of a graph $\Gamma$ is said
to be {\it cyclic} if it admits an automorphism cyclically permuting all of the vertices of $\Gamma$.

\begin{theorem}\label{graph}
If there exists a $(v,2;k_1,k_2;\lambda,\lambda)$-GSEDF, then the graph $2\lambda K_v$ has a cyclic $K_{k_1,k_2}$-decomposition.
\end{theorem}

\begin{proof}

Suppose $\{D_1, D_2\}$ is a $(v,2;k_1,k_2;\lambda,\lambda)$-GSEDF in a group $G$ of order $v$.  We use $(D_1; D_2)$ to denote the complete bipartite graph $K_{k_1,k_2}$ with vertex set $D_1\cup D_2$ and edge set $\{( d_1, d_2): d_1\in D_1, d_2\in D_2\}$. Let ${\cal H}=\{(D_1+g; D_2+g):g\in G\}$. Then it is easy to check that ${\cal H}$ is  a cyclic $K_{k_1,k_2}$-decomposition of $2\lambda K_v$.
\end{proof}

Not much result is known for cyclic $K_{a,b}$-decompositions of complete multigraphs, see \cite{CD}. By Theorems~\ref{C1}, \ref{2n}-\ref{4m-1} and \ref{graph}, we have the following results:
 \begin{description}
   \item[(1)] Let $v=ab+1$. Then $2K_{v}$ has a cyclic $K_{a,b}$-decomposition;
   \item[(2)] Let $v=p_1p_2\dots p_n$, $p_i>1$, $1\le i\le n$, and $p_1, p_2,\dots,p_n$ be odd integers. Then $2^{n}K_v$ has a cyclic $K_{2^{n},(v-1)/2}$-decomposition;
   \item[(3)]$6K_{16}$ has a cyclic $K_{5,9}$-decomposition;
   \item[(4)] Let $q$ and $q+2$ be odd prime powers and $v=q(q+2)$. Then  $\frac {v+1} 2 K_v$ has a cyclic $K_{\frac{v-1}2,\frac {v+1} 2}$-decomposition;
   \item[(5)] Let $t=p_1p_2\dots p_n$, $p_i>1$, $1\le i\le n$, where $p_1, p_2,\dots,p_n$ are odd integers.
       For any prime power $q$ with $q\equiv1\pmod 4$, then $(q-1)2^{n-1}K_{qt}$ has a cyclic $K_{(q-1)2^{n-1},\frac{qt-1}{2}}$-decomposition;
   \item[(6)] Let $t=1, n=0$; or $t=p_1p_2\dots p_n$, $p_i>1$, $1\le i\le n$, where $p_1, p_2,\dots,p_n$ are odd integers. If $4m-1$ is a prime power or $4m-1=q(q+2)$ where both $q$ and $q+2$ are prime powers, then $(m\times 2^{n+1})K_{(4m-1)t}$ has a cyclic $K_{m\times 2^{n+1},\frac {4mt-t-1} 2}$-decomposition.
 \end{description}

\section{Concluding Remarks}

In this paper, we have constructed some new GSEDFs for $m=2$. But there is still a long way to go before the existence of a  $(v,2;k_1,k_2;\lambda,\lambda)$-GSEDF can be determined completely.

\begin{remark}
For each $v\leq 21$ we have determined the existence of a $(v,2;k_1,k_2;\lambda,\lambda)$-GSEDF as below.
When $\lambda=1$, a $(v,2;k_1,k_2;\lambda, \lambda)$-GSEDF exists by Theorem \ref{C1}. When $\lambda\geq2$, all the possible  parameters $(v,k_1,k_2,\lambda)$ of a $(v,2;k_1,k_2;\lambda,\lambda)$-GSEDF are listed as follows:
$(v, k_1, k_2, \lambda)\in M_1\cup M_2\cup M_3\cup M_4\cup M_5$, where\\
$M_1= \{(21,4,10,2),\ (21,8,10,4)\}$,\\
$M_2=\{(15,4,7,2),\ (16,5,9,3)\}$,\\
$M_3= \{(13,4,9,3),\ (15,7,8,4),\ (16,6,10,4),\ (21,5,16,4)\}$,\\
$M_4= \{(7,3,4,2),\ (9,4,4,2),\ (11,5,6,3),\ (13,6,6,3), (17,8,8,4),\ (19,9,10,5)\}$,\\
$M_5= \{(10,3,6,2),\ (11,4,5,2),\ (13,3,8,2),\ (13,4,6,2),\ (15,6,7,3),\ (16,3,10,2),\ (16,5,6,2)$,

~~~$(17,4,8,2),\ (17,4,12,3),\ (17,6,8,3),\ (19,3,12,2),\ (19,4,9,2),\ (19,6,6,2),\ (19,6,9,3)$,

~~$(19,8,9,4),\ (21,5,8,2),\ (21,4,15,3),\ (21,5,12,3),\ (21,6,10,3),\ (21,10,10,5)\}$.

By Theorems \ref{DS}, \ref{power} and \ref{4m-1}, we know the GSEDFs with parameters in $M_1$, $M_3$ and $M_4$ all exist. By Theorems \ref{2n} and  \ref{g16}, the GSEDFs with parameters in $M_2$ exist. From \cite{JL}, there does not exist a GSEDF with parameters $(21,10,10,5)$. The existence of the other GSEDFs with parameters in $M_5$ is denied by a computer exhaustive search.
\end{remark}

In Section 3, we have  given some nonexistence results for GSEDFs and proved the nonexistence of a $(v,3;k_1,k_2,k_3;\lambda_1,\lambda_2,\lambda_3)$-GSEDF with $k_1+k_2+k_3<v$. For the existence of a $(v,4;k_1,k_2,k_3,k_4;\lambda_1,\lambda_2,\lambda_3,\lambda_4)$-GSEDF with $\sum\limits_{i=1}^{4}k_i< v$, we conjecture it does not exist.
\begin{remark}
  If a $(v,4;k_1,k_2,k_3,k_4;\lambda_1,\lambda_2,\lambda_3,\lambda_4)$-GSEDF exists, then by Lemma \ref{equ} we can find a set $\{\alpha_1,\alpha_2,\dots,\alpha_m\}$ such that $\alpha_1+\alpha_2+\dots+\alpha_m=1$, where $\alpha_j\in\{\alpha_j^{+},\alpha_j^{-}\}$, $j=1,2,\dots,m$. Let $J=\{j: \alpha_j=\alpha_j^+, 1\le j\le m\}$. It is easy to see when $|J|=0,~1,~m-1,~m$, the equation $\alpha_1+\alpha_2+\dots+\alpha_m=1$ is impossible. So when $m=4$, we only need to prove when $|J|=2$, the equation $\alpha_1+\alpha_2+\dots+\alpha_m=1$ is impossible. That is
 \begin{equation}\label{alpha}
  \alpha_i^{+}+\alpha_j^{+}+\alpha_k^{-}+\alpha_t^{-}=1,
  \end{equation}
   where $i$, $j$, $k$, $t$ take different values in $\{1,2,3,4\}$. However, we can not determine whether the equation \eqref{alpha} is impossible at present.
\end{remark}





\end{document}